\documentclass[11pt]{amsart}
\usepackage{array, amsmath, enumerate, color}
\usepackage{amssymb}
\usepackage{graphicx,subfigure}

\makeatletter
 \usepackage{fullpage} 
 \usepackage{amsthm}

\usepackage{fullpage}

\makeatother

\begin{document}

\newtheorem{thm}{Theorem}[section]
\newtheorem{lem}[thm]{Lemma}
\newtheorem{prop}[thm]{Proposition}
\newtheorem{cor}[thm]{Corollary}
\newtheorem{con}[thm]{Conjecture}
\newtheorem{claim}[thm]{Claim}
\newtheorem{obs}[thm]{Observation}
\newtheorem{ques}[thm]{Question}
\newtheorem{definition}[thm]{Definition}
\newtheorem{example}[thm]{Example}
\newtheorem{rmk}[thm]{Remark}
\newcommand{\di}{\displaystyle}
\def\dfc{\mathrm{def}}
\def\cF{{\cal F}}
\def\cH{{\cal H}}
\def\cT{{\cal T}}
\def\C{{\mathcal C}}
\def\cA{{\cal A}}
\def\cB{{\mathcal B}}
\def\P{{\mathcal P}}
\def\Q{{\mathcal Q}}
\def\cP{{\mathcal P}}
\def\cp{\alpha'}
\def\Frk{F_k^{2r+1}}

\newcommand{\F}{\mathcal{F}}

\title{Planar graphs with girth at least 5 are $(3,4)$-colorable}

\author{Ilkyoo Choi$^1$ \and  Gexin Yu$^2$ \and Xia Zhang$^{3,*}$}

\address{$^1$
Department of Mathematics, Hankuk University of Foreign Studies, Yongin-si, Gyeonggi-do, Republic of Korea, ilkyoo@hufs.ac.kr.}

\address{$^2$Department of Mathematics\\
The College of William and Mary\\
Williamsburg, VA, 23185, gyu@wm.edu.}

\address{$^3$School of Mathematics and Statistics\\
Shangdong Normal University\\
Jinan, China, xiazhang@sdnu.edu.cn, pandarhz@sina.com
}

\thanks{$*$Corresponding author. \\
The first author's research was supported by the Basic Science Research Program through the National Research Foundation of Korea (NRF) funded by the Ministry of Education (NRF-2018R1D1A1B07043049), and also by the Hankuk University of Foreign Studies Research Fund. The research of the second author is supported in part by the NSA grant  H98230-16-1-0316 and the Natural Science Foundation of China (11728102). The last author's research is supported by the Shandong Provincial Natural Science Foundation (No. ZR2019MA032, ZR2014JL001, ZR2016AQ01), the National Natural Science Foundation of China (No.11701342) and the Excellent Young Scholars Research Fund of Shandong Normal University of China.}


\maketitle

\begin{abstract}
A graph is $(d_1, \ldots, d_k)$-colorable if its vertex set can be partitioned into $k$ nonempty subsets so that the subgraph induced by the $i$th part has maximum degree at most $d_i$ for each $i\in\{1, \ldots, k\}$.
It is known that for each pair $(d_1, d_2)$, there exists a planar graph with girth $4$ that is not $(d_1, d_2)$-colorable.
This sparked the interest in finding the pairs $(d_1, d_2)$ such that planar graphs with girth at least $5$ are $(d_1, d_2)$-colorable.
Given $d_1\leq d_2$, it is known that planar graphs with girth at least $5$ are $(d_1, d_2)$-colorable if either $d_1\geq 2$ and $d_1+d_2\geq 8$ or $d_1=1$ and $d_2\geq 10$.
We improve an aforementioned result by providing the first pair $(d_1, d_2)$ in the literature satisfying $d_1+d_2\leq 7$ where planar graphs with girth at least $5$ are $(d_1, d_2)$-colorable.
Namely, we prove that planar graphs with girth at least $5$ are $(3, 4)$-colorable.
\end{abstract}


\section{Introduction}

All graphs in this paper are finite and simple, which means no loops and no multiple edges.
For an integer $k$, let $[k]=\{1, \ldots, k\}$.
Given a graph $G$, let $V(G)$ and $E(G)$ denote its vertex set and edge set, respectively.
A graph is {\it $(d_1, \ldots, d_k)$-colorable} if its vertex set can be partitioned into $k$ nonempty subsets so that the subgraph induced by the $i$th part has maximum degree at most $d_i$ for each $i\in [k]$.
This notion is known as {\it improper coloring}, or {\it defective coloring}, and has recently attracted much attention.
Improper coloring is a relaxation of the traditional proper coloring, however, it also opens up an opportunity to gain refined information on partitioning the graph compared to the traditional proper coloring.

The Four Color Theorem~\cite{1977ApHa,1977ApHaKo} states that the vertex set of a planar graph can be partitioned into four independent sets; this means that every planar graph is $(0, 0, 0, 0)$-colorable since an independent set induces a graph with maximum degree at most $0$.
A natural question to ask is what happens when we try to partition the vertex set of a planar graph into fewer parts.
Already in 1986, Cowen, Cowen, and Woodall~\cite{1986CoCoWo} proved that a planar graph is $(2, 2, 2)$-colorable.
The previous result is sharp since Eaton and Hull~\cite{1999EaHu} and independently \v Skrekovski~\cite{1999Sk} both acknowledged the existence of a planar graph that is not $(1, h, l)$-colorable for any given $h$ and $l$; for an explicit construction see~\cite{unpub_ChEs}.
Hence, improper coloring of a planar graph with no restriction is completely solved for $k\geq 3$.

Since sparser graphs are easier to color, a natural direction of research is to consider sparse planar graphs, and a popular sparsity condition is imposing a restriction on girth.
Gr\"otzsch's theorem~\cite{1959Gr} states that a planar graph with girth at least $4$ is $(0, 0, 0)$-colorable.
Therefore it only remains to consider partitioning the vertex set of a planar graph into two parts.
Moreover, since there exists a planar graph with girth $4$ that is not $(d_1, d_2)$-colorable for each pair $(d_1, d_2)$ (see~\cite{2015MoOc} for an explicit construction), there has been a considerable amount of research towards improper coloring planar graphs with girth at least $5$.
For various results regarding improper coloring planar graphs with girth at least $6$ or other sparse graphs that are not necessarily planar, see~\cite{2010BoIvMoOcRa,2013BoKoYa,2011BoKo,2014BoKo,2006HaSe,2015MoOc}.
Similar research has also been done for graphs on surfaces as well~\cite{unpub_ChEs}.

In this paper, we focus on planar graphs with girth at least $5$.
\v Skrekovski~\cite{2000Sk} showed that planar graphs with girth at least $5$ are $(4, 4)$-colorable and Borodin and Kostochka~\cite{2014BoKo} proved a result that implies planar graphs with girth at least $5$ are $(2, 6)$-colorable.
Answering a question by Raspaud, Choi and Raspaud~\cite{2015ChRa} proved that planar graphs with girth at least $5$ are $(3, 5)$-colorable.
Recently, Choi et al.~\cite{2017ChChJeSu} proved that planar graphs with girth at least $5$ are $(1, 10)$-colorable, which answered a question by Montassier and Ochem~\cite{2015MoOc} in the affirmative.
By a construction of Borodin et al.~\cite{2010BoIvMoOcRa}, it is also known that planar graphs with girth at least $5$ (even $6$) are not necessarily $(0, d)$-colorable for an arbitrary $d$.
As a conclusion, there are only finitely many pairs $(d_1, d_2)$ that are unknown for which planar graphs with girth at least $5$ are $(d_1, d_2)$-colorable.
To sum up, all previous knowledge about improper coloring planar graphs with girth at least $5$ is the following:

\begin{thm}
Given $d_1\leq d_2$, planar graphs with girth at least $5$ are $(d_1, d_2)$-colorable if
\begin{enumerate}
\item $d_1\geq 2$ and $d_1+d_2\geq 8$~\cite{2014BoKo,2015ChRa,2000Sk} or
\item $d_1=1$ and $d_2\geq 10$~\cite{2017ChChJeSu}.
\end{enumerate}
\end{thm}

In this paper, we prove the following theorem, which reveals the first pair $(d_1, d_2)$ satisfying $d_1+d_2\leq 7$ where planar graphs with girth at least $5$ are $(d_1, d_2)$-colorable.

\begin{thm}\label{thm:main}
Planar graphs with girth at least $5$ are $(3,4)$-colorable.
\end{thm}

The above theorem also improves the best known answer to the following question, which was explicitly stated in~\cite{2015ChRa}:

\begin{ques}[\cite{2015ChRa}]
What is the minimum $d^3_2$ such that planar graphs with girth at least $5$ are $(3, d^3_2)$-colorable?
\end{ques}

Since Montassier and Ochem~\cite{2015MoOc} constructed a planar graph with girth $5$ that is not $(3, 1)$-colorable, along with Theorem~\ref{thm:main}, this shows that $d^3_2\in\{2, 3, 4\}$.
Theorem~\ref{thm:main} is an improvement to the previously best known bound, which was by Choi and Raspaud~\cite{2015ChRa}.
It would be remarkable to determine the exact value of $d^3_2$.

Section 2 will reveal some structural properties of a minimum counterexample to Theorem~\ref{thm:main}.
In Section 3, we will show that a minimum counterexample to Theorem~\ref{thm:main} cannot exist via discharging, hence proving the theorem.

We end the introduction with some definitions that will be used throughout the paper.
A {\it $d$-vertex}, a {\it $d^-$-vertex}, and a {\it $d^+$-vertex} is a vertex of degree $d$, at most $d$, and at least $d$, respectively.
A {\it $d$-neighbor} of a vertex is a neighbor that is a $d$-vertex.
A $d$-vertex is a {\em poor $d$-vertex} (or {\it $dp$-vertex}) and a {\em semi-poor $d$-vertex} (or {\it $ds$-vertex}) if it has exactly one and two, respectively, $3^+$-neighbors; otherwise, it is called a {\em rich vertex} (or {\em $dr$-vertex}).
A {\em $dr^+$-vertex} is a rich $d^+$-vertex.  A $ds^+$-$vertex$ is a $d^+$-vertex with at least two $3^+$-neighbors. A $dp^-$-$vertex$ is a $d^-$-vertex with at most one $3^+$-neighbor.
An edge $uv$ is a {\em heavy edge} if both $u$ and $v$ are $5^+$-vertices, and neither $u$ nor $v$ is a $5p$-, $5s$-, or $6p$-vertex.

Throughout the paper, let $G$ be a counterexample to Theorem~\ref{thm:main} with the minimum number of $3^+$-vertices, and subject to that choose one with the minimum number of $|V|+|E|$.
It is easy to see that $G$ must be connected and there are no $1$-vertices in $G$.
From now on, given a (partially) $(3, 4)$-colored graph, let $i$ be the color of the color class where maximum degree $i$ is allowed for $i\in\{3, 4\}$.
We say a vertex with color $i$ is {\it $i$-saturated} if it already has $i$ neighbors of the same color.
A vertex is {\it saturated} if it is either $3$-saturated or $4$-saturated.

\section{Structural Lemmas}

In this section, we reveal useful structural properties of $G$.

\begin{lem}\label{lem:edge}
Every edge $xy$ of $G$ has an endpoint with degree at least $5$.
\end{lem}

\begin{proof}
Suppose to the contrary that $x$ and $y$ are both $4^-$-vertices. Since $G-xy$ is a graph with fewer edges than $G$ and the number of $3^+$-vertices did not increase, there is a $(3, 4)$-coloring $\varphi$ of $G-xy$. If either $\varphi(x)\neq\varphi(y)$ or $\varphi(x)=\varphi(y)=4$, then $\varphi$ is also a $(3, 4)$-coloring of $G$. Otherwise, $\varphi(x)=\varphi(y)=3$, and at least one of $x, y$ is $3$-saturated in $G-xy$.  For one $3$-saturated vertex in $\{x, y\}$, we may recolor it with the color $4$, since all of its neighbors have color $3$ in $G$.  In all cases we end up with a $(3,4)$-coloring of $G$, which is a contradiction.
\end{proof}

\begin{lem}\label{lem:3vx}
There is no $3$-vertex in $G$.
\end{lem}
\begin{proof}
Suppose to the contrary that $v$ is a $3$-vertex of $G$ with neighbors $v_1, v_2, v_3$. By Lemma~\ref{lem:edge}, we know that $v_1, v_2, v_3$ are $5^+$-vertices. Obtain a graph $H$ from $G-v$ by adding paths $v_1u_1v_2, v_2u_2v_3, v_3u_3v_1$ of length two between the neighbors of $v$. See Figure~\ref{fig:3vx} for an illustration. Note that $H$ is planar and still has girth at least $5$ since the pairwise distance between $v_1, v_2, v_3$ did not change.  Since $H$ has fewer $3^+$-vertices than $G$, there is a $(3, 4)$-coloring $\varphi$ of $H$.

\begin{figure}[ht]
\begin{center}
\includegraphics[scale=0.7]{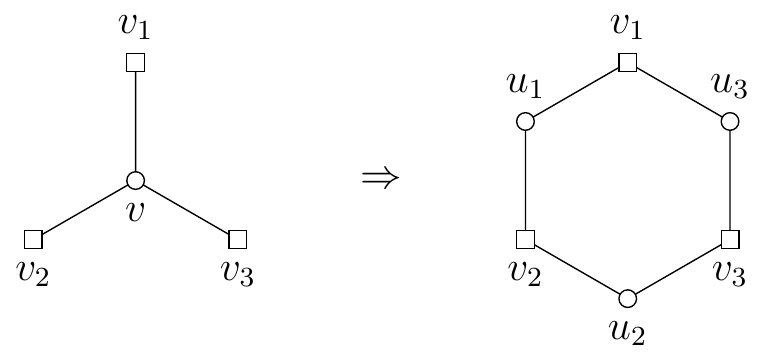}
\caption{Obtaining $H$ from $G$ in Lemma~\ref{lem:3vx}.}
\label{fig:3vx}
\end{center}
\end{figure}

Without loss of generality, we may assume $\varphi(u_1)=\varphi(u_2)$.
Since each of $v_1, v_2, v_3$ has a neighbor in $\{u_1, u_2\}$, using the color $\varphi(u_1)$ on $v$ gives a $(3, 4)$-coloring of $G$, which is a contradiction.
\end{proof}

\begin{lem}\label{vertex-degree}
If $v$ is an $8^-$-vertex of $G$, then in every $(3, 4)$-coloring of $G-v$, $v$ has a saturated neighbor in $G-v$ that cannot be recolored.    In particular,
\begin{enumerate}[$(i)$]
\item if $d(v)=2$, then for each $i\in\{3, 4\}$,  $v$ has an $i$-saturated $(i+2)^+$-neighbor $u$ that cannot be recolored.
Moreover, if $u$ is an $8^-$-vertex, then $u$ has a $j$-saturated $(j+2)^+$-neighbor where $\{i, j\}=\{3, 4\}$.
\item if $d(v)\in\{ 4, 5\}$, then $v$ has a $4$-saturated neighbor that is either a $9^+$-vertex or a $6s^+$-vertex.
\item if $d(v)\in\{6, 7, 8\}$, then $v$ has a saturated neighbor that is either a $9^+$-vertex or a $5s^+$-vertex.
\end{enumerate}
\end{lem}

\begin{proof}
Since $G-v$ is a graph with fewer edges than $G$ and the number of $3^+$-vertices did not increase, there exists a $(3,4)$-coloring $\varphi$ of $G-v$.  Note that for each $i\in\{3, 4\}$, since letting $\varphi(v)=i$ cannot be a $(3, 4)$-coloring of $G$,  $v$ has either an $i$-saturated neighbor or $i+1$ neighbors with the color $i$.  Since $v$ is an $8^-$-vertex, $v$ cannot have both four neighbors of color $3$ and five neighbors of color $4$.  Let $j\in\{3, 4\}$ such that $v$ has at most $j$ neighbors with color $j$, one of which is $j$-saturated.
If every $j$-saturated neighbor of $v$ can be recolored, then we can color $v$ with $j$, a contradiction.  Hence, $v$ must have at least one $j$-saturated neighbor that cannot be recolored.

Let $u$ be a non-recolorable $j$-saturated neighbor of $v$ and let $\{i, j\}=\{3, 4\}$.
We know that $u$ is a $(j+2)^+$-vertex, because it is adjacent to $v$, $j$ neighbors colored with $j$, and at least one neighbor $x$ colored with $i$ (since $u$ cannot be recolored with $i$).
Moreover, if $d(u)\leq 8$, then $x$ must be $i$-saturated.
In particular,

\begin{enumerate}[$(i)$]
\item if $d(v)=2$, then $v$ has both a non-recolorable $3$-saturated neighbor and a non-recolorable $4$-saturated neighbor. For $j\in \{3,4\}$, the $j$-saturated neighbor has degree at least $j+2$, and if its degree is at most $8$, then it has an $i$-saturated neighbor of degree at least $i+2$, where $\{i,j\}=\{3,4\}$.

\item if $d(v)\in\{4, 5\}$, then $v$ must have a non-recolorable $4$-saturated neighbor $u$. So $u$ is either a $9^+$-vertex or a $6s^+$-vertex.

\item if $d(v)\in\{6,7, 8\}$, then $u$ must be either a $9^+$-vertex or a $5s^+$-vertex.
\end{enumerate}
This finishes the proof of this lemma.
\end{proof}

\begin{lem}\label{lem:6face}
Let $C$ be a $6$-face
$u_1u_2u_3u_4u_5u_6$ of $G$.
\begin{enumerate}[$(a)$]
\item If $C$ contains three $2$-vertices and a $5$-vertex, then the other two vertices are $7^+$-vertices.

\item If $C$ contains exactly two $2$-vertices, then $C$ contains at most two $5p$-vertices.
Moreover,

\begin{enumerate}[$(b1)$]
\item  if $C$ contains exactly one $5p$-vertex, then it contains at most two of $5s$-vertices and $6p$-vertices;
\item if $C$ contains two $5p$-vertices, then either $C=F_{6a}$ (see Figure~\ref{figure-special-6-face}) or it contains neither $5s$-vertices nor $6p$-vertices.
\end{enumerate}

\item If $C$ contains exactly one $2$-vertex, then it contains at most one $5p$-vertex. Moreover,

\begin{enumerate}[$(c1)$]
\item if $C$ contains exactly one $5p$-vertex, then it contains at most two of $5s$-vertices and $6p$-vertices;
\item if $C$ contains no $5p$-vertices, then it contains at most four of $5s$-vertices and $6p$-vertices.
\end{enumerate}

\item If $C$ contains no $2$-vertex, then it contains no poor vertices and at most four $5s$-vertices.
\end{enumerate}
\end{lem}

\begin{proof}
Note that by Lemma~\ref{lem:edge}, no two $2$-vertices are adjacent to each other.
We will show that if $C$ is not one of the above, then we can obtain a $(3, 4)$-coloring of $G$, which is a contradiction.

{\bf{ (a)}}: Let $u_1, u_3, u_5$ be the $2$-vertices and let $u_4$ be a $5$-vertex of $C$.   By Lemma~\ref{vertex-degree} $(i)$, both $u_2$ and $u_6$ are $6^+$-vertices, so without loss of generality, suppose to the contrary that $u_6$ is a $6$-vertex.  Since $G-u_5$ is a graph with fewer edges than $G$ and the number of $3^+$-vertices does not increase, there is a $(3, 4)$-coloring $\varphi$ of $G-u_5$.
By Lemma~\ref{vertex-degree} $(i)$, we know $u_4$ is $3$-saturated and has a $4$-saturated $6^+$-neighbor  and $u_6$ is $4$-saturated and has a $3$-saturated $5^+$-neighbor. Hence, $\varphi(u_3)=3$ and $\varphi(u_1)=4$.

If $\varphi(u_2)=3$, then recolor $u_3$ with $4$ and color $u_5$ with $3$ to obtain a $(3, 4)$-coloring of $G$.
If $\varphi(u_2)=4$, then recolor $u_1$ with $3$ and color $u_5$ with $4$ to obtain a $(3, 4)$-coloring of $G$.

{\bf (b)}: Note that each $5p$-vertex on $C$ must have a $2$-neighbor on $C$, and by Lemma~\ref{vertex-degree} $(i)$,  each 2-vertex has at most one $5p$-neighbor. So $C$ contains at most two $5p$-vertices because it has exactly two $2$-vertices.


{\bf (b1)} Assume that $u_1$ is the unique $5p$-vertex on $C$.  By Lemma~\ref{vertex-degree} $(ii)$, none of $u_2,u_6$ is a $5s$- or $6p$-vertex. If $u_4$ is neither a $5s$-vertex nor a $6p$-vertex, then $C$ contains at most two of $5s$-vertices and $6p$-vertices. If $u_4$ is a $6p$-vertex, then either $u_3$ or $u_5$ is a $2$-vertex, so again $C$ contains at most two of $5s$-vertices and $6p$-vertices. If $u_4$ is a $5s$-vertex, then by Lemma~\ref{vertex-degree} $(ii)$, one of $u_3$ and $u_5$ must be
a $6s^+$-vertex, a $9^+$-vertex, or a $2$-vertex.
Therefore, $C$ contains at most two of $5s$-vertices and $6p$-vertices.

{\bf(b2)}
Now assume that $C$ contains two $5p$-vertices. Observe that if $u_1, u_4$ are the two $5p$-vertices on $C$, then by Lemma~\ref{vertex-degree} $(ii)$, none of $u_2,u_3,u_5,u_6$ is a $5s$-vertex or a $6p$-vertex, as claimed.
Therefore, we may assume that $u_1, u_3$ are the two $5p$-vertices on $C$.

Note that $u_2$ cannot be a $2$-vertex by Lemma~\ref{vertex-degree} $(i)$. So both $u_4$ and $u_6$ are 2-vertices. By Lemma~\ref{vertex-degree} $(i)$ and $(ii)$, both $u_2$ and $u_5$ are $6^+$-vertices. We may assume that $u_5$ is a $6p$-vertex, for otherwise $C$ contains neither $5s$-vertices nor $6p$-vertices. Assume that $C$ is not a special 6-face $F_{6a}$, which implies that $u_2$ is a $6$-vertex.
By Lemma~\ref{vertex-degree} $(i)$, in a $(3,4)$-coloring $\varphi$ of $G-u_6$, we know $u_1$ is $3$-saturated and $u_5$ is 4-saturated and both are non-recolorable. It follows that $u_2$ is $4$-saturated, $u_4$ is colored with $4$ and non-recolorable, and furthermore $u_3$ is 3-saturated.
Now we can recolor $u_4, u_3, u_2, u_1$ with $3, 4, 3, 4$ respectively, and  color $u_6$ with $3$ to obtain a $(3, 4)$-coloring of $G$.

{\bf (c)}: Let $u_1$ be the unique $2$-vertex on $C$.
A $5p$-vertex must have a $2$-neighbor on $C$, and by Lemma~\ref{vertex-degree} $(i)$, a $2$-vertex has at most one $5p$-neighbor, so $C$ contains at most one $5p$-vertex.

{\bf (c1)} Assume that $C$ has one $5p$-vertex $u_2$. By Lemma~\ref{vertex-degree} $(i)$ and $(ii)$, $u_6$ cannot be a $5$-vertex, and $u_3$ cannot be a $5s$-vertex or a $6p$-vertex.
If $u_6$ is not a $6p$-vertex, then $C$ has at most two of $5s$-vertices and $6p$-vertices.
If $u_6$ is a $6p$-vertex, then $u_4$ and $u_5$ cannot be both $5s$-vertices by Lemma~\ref{vertex-degree} $(ii)$.
Note that either $u_4$ or $u_5$ cannot be $6p$-vertices since $C$ has only one $2$-vertex $u_1$.

{\bf (c2)} Now assume that $C$ contains no $5p$-vertices.
Consider three consecutive vertices $u_{i-1}, u_i, u_{i+1}$ on $C$.
If $u_i$ is a $6p$-vertex, then either $u_{i-1}$ or $u_{i+1}$ must be a $2$-vertex.
If $u_i$ is a $5s$-vertex, then by Lemma~\ref{vertex-degree} $(ii)$,
either $u_{i-1}$ or $u_{i+1}$ is
a $6s^+$-vertex, a $9^+$-vertex, or a $2$-vertex.
Therefore, $C$ contains at most four of $5s$-vertices and $6p$-vertices.


{\bf (d)}:   If $C$ contains no 2-vertex, then it contains neither a $5p$-vertex nor a $6p$-vertex. By Lemma~\ref{vertex-degree} $(ii)$, a $5$-vertex must have a $6^+$-neighbor, so the two $3^+$-neighbors of a $5s$-vertex cannot be both $5s$-vertices. Therefore, $C$ contains at most four $5s$-vertices.
\end{proof}


\begin{lem}\label{lem:F2}
If $F_{6b}$ is a $6$-face with three $2$-vertices and three $6p$-vertices (see  Figure~\ref{figure-special-6-face}), then $F_{6b}$ cannot share an edge with a $5$-face with two $2$-vertices.
\end{lem}

\begin{proof}
Let $C=u_1\ldots u_6$ be an $F_{6b}$ with three $2$-vertices $u_1, u_3, u_5$ and three $6p$-vertices.
Note that two $2$-vertices cannot be adjacent to each other by Lemma~\ref{lem:edge}.
Suppose to the contrary that a $5$-face $C'$ shares an edge with $C$. Then $C$ and $C'$ share exactly two edges and without loss of generality, assume that $C'=u_6u_1u_2v_1v_2$ and, by symmetry, we may assume that $v_1$ is a $2$-vertex.

Since $G-u_1$ is a graph with fewer edges than $G$ and the number of $3^+$-vertices did not increase, there is a $(3, 4)$-coloring $\varphi$ of $G-u_1$.  By Lemma~\ref{vertex-degree} $(i)$, both $u_2$ and $u_6$ are non-recolorable and one of $u_2$ and $u_6$ is $3$-saturated and the other is $4$-saturated.

First assume that $u_6$ is $3$-saturated and $u_2$ is $4$-saturated.
Since $u_2$ is a $6$-vertex, by Lemma~\ref{vertex-degree} $(i)$, $u_2$ must have exactly one $3$-saturated neighbor and all other neighbors are colored with the color $4$.
In particular, $\varphi(v_1)=4$.
Also, by Lemma~\ref{vertex-degree} $(i)$, $u_6$ has a $4$-saturated neighbor, which must be $v_2$. Hence, we can recolor $v_1$ with the color $3$ and color $u_1$ with the color $4$ to obtain a $(3, 4)$-coloring of $G$, which is a contradiction.

Now assume that $u_6$ is $4$-saturated and $u_2$ is $3$-saturated.  
By Lemma~\ref{vertex-degree} $(i)$, $u_6$ must have a $3$-saturated neighbor, which must be $v_2$, and all other neighbors are colored with the color 4. In particular, $\varphi(u_5)=4$. 
Also, by Lemma~\ref{vertex-degree}, we know that $u_2$ must have a $4$-saturated neighbor, which is neither $u_3$ nor $v_1$.
If $\varphi(v_1)=3$, then we can recolor $v_1$ with the color $4$ and color $u_1$ with the color $3$ to obtain a $(3, 4)$-coloring of $G$, which is a contradiction.
Therefore, $\varphi(v_1)=4$, which further implies that $\varphi(u_3)=3$.  Now, if we can recolor $u_3$ with the color $4$, then we can color $u_1$ with the color $3$ to obtain a $(3, 4)$-coloring of $G$, which is a contradiction.
Hence, $u_4$ must be $4$-saturated, and in particular $\varphi(u_4)=4$. Finally, we can recolor $u_5$ with the color $3$ and color $u_1$ with $4$ to obtain a $(3, 4)$-coloring of  $G$, which is a contradiction.
\end{proof}

\begin{lem}\label{lem:5-face-1-2vtx}
If $C$ is a $5$-face $u_1\ldots u_5$ with exactly one $2$-vertex $u_1$, then either
\begin{itemize}
\item $C$ contains at most two of $5p$-, $5s$-, and $6p$-vertices, or
\item $C$ is a special $5$-face $F_{5c}$ or $F_{5d}$ in Figure~\ref{figure-special-6-face}.
\end{itemize}
\end{lem}


\begin{proof}
Assume that $C$ contains at least three $5p$-, $5s$-, and $6p$-vertices. By symmetry, we may assume that $u_3$ is a $5s$-vertex. Note that by Lemma~\ref{vertex-degree},  $u_2$ is not a $5p$-vertex, and $u_3$ has a $6s^+$-neighbor or $9^+$-neighbor, which is either $u_2$ or $u_4$.  If $u_2$ is a $6s^+$-vertex or $9^+$-vertex, then both $u_3$ and $u_4$ are $5s$-vertices, so by Lemma~\ref{vertex-degree}, both $u_2$ and $u_5$ are $6s^+$-vertices or $9^+$-vertices, which is a contradiction.
Therefore we may assume that $u_4$ is a $6s^+$-vertex or $9^+$-vertex.
Now $u_2$ is a $5s$-vertex or $6p$-vertex, and $u_5$ is a $5p$-, $5s$-, or $6p$-vertex.

First assume that $u_2$ is a $5s$-vertex.
Since $G-u_1$ is a graph with fewer edges than $G$ and the number of $3^+$-vertices did not increase, there is a $(3, 4)$-coloring $\varphi$ of $G-u_1$.
By Lemma~\ref{vertex-degree} $(i)$, $u_2$ must be $3$-saturated and $u_5$ must be a $4$-saturated $6p$-vertex.
This further implies that $u_4$ is $3$-saturated.
Note that $u_2$ must have a $4$-saturated neighbor and three neighbors of color $3$.
Since $\varphi(u_4)=3$, we know $u_3$ cannot be the $4$-saturated neighbor of $u_2$, so $\varphi(u_3)=3$.
Now, since $u_3$ has neither five neighbors colored with the color $4$ nor a $4$-saturated neighbor, $u_3$ can be recolored with $4$.
Now, by recoloring $u_3$ with the color $4$ and coloring $u_1$ with the color $3$, we obtain a $(3, 4)$-coloring of $G$, which is a contradiction.

Now assume that $u_2$ is a $6p$-vertex.  Let $u$ be a 2-neighbor of $u_3$ that is not on $C$.
Since $G-u$ is a graph with fewer edges than $G$ and the number of $3^+$-vertices did not increase, there is a $(3, 4)$-coloring $\varphi$ of $G-u$.
By Lemma~\ref{vertex-degree} $(ii)$, $u_3$ is $3$-saturated and $u_3$ has a  $4$-saturated $6^+$-neighbor $x$.
If $x=u_2$, then we can recolor $u_2$ with the color 3, and color $u_3$ and $u$ with the colors 4 and 3, respectively, to obtain a $(3, 4)$-coloring of $G$, which is a contradiction.

Therefore $x=u_4$, which implies that $\varphi(u_4)=4$ and $\varphi(u_2)=3$.
Since recoloring $u_2$ with the color $4$ must not be possible, we know that all neighbors of $u_2$, except $u_3$, are colored with the color $4$.
In particular, $\varphi(u_1)=4$ and $u_1$ is non-recolorable.
This further implies that $u_5$ is $3$-saturated and non-recolorable.
Now, $u_4$ must be $4$-saturated and non-recolorable. That is to say, $u_4$ must have four neighbors colored with $4$. Moreover, $u_4$ must have either a $3$-saturated neighbor other than $u_3$, $u_5$, or at least four neighbors other than $u_3$ colored with $3$.
Hence, $u_4$ is a $7r^+$-vertex or $9s^+$-vertex, that is, $C$ is either $F_{5c}$ or $F_{5d}$.
\end{proof}


\begin{lem}\label{lem:7-face}
If $F$ is a $7$-face, then one of the following is true:
\begin{itemize}
\item $F$ has at most six $2$-, $5p$-, $5s$-, or $6p$-vertices;
\item $F$ has at least two $5s$-vertices;
\item $F$ is a special $7$-face $F_7$ (see Figure~\ref{figure-special-6-face}).
\end{itemize}
\end{lem}

\begin{proof}
Note that two $2$-vertices cannot be adjacent to each other by Lemma~\ref{lem:edge}.
Suppose to the contrary that $F$ contains seven of $2$-, $5p$-, $5s$-, and $6p$-vertices, and at most one $5s$-vertex. Denote the vertices around $F$ by $u_1,u_2,\ldots, u_7$ in order.
Without loss of generality, we may assume that one vertex $u_1$ is a $5s$-vertex, for otherwise, two $6p^-$-vertices would be adjacent to each other, which contradicts Lemma~\ref{vertex-degree} $(ii)$ and $(iii)$.  All other vertices of $F$ are $2$-vertices and $6p^-$-vertices.

Without loss of generality, we may assume that $u_2, u_4, u_6$ are $2$-vertices and $u_3, u_5, u_7$ are $6p^-$-vertices.
Since $u_2$ is a $2$-vertex, by Lemma~\ref{vertex-degree} $(i)$, we know that $u_3$ is a $6p$-vertex.
Since a $5p$-vertex cannot have a $5s$-neighbor by Lemma~\ref{vertex-degree} $(ii)$, we know that $u_7$ must be a $6p$-vertex.
If $u_5$ is a $6p$-vertex, then $F$ is a special face $F_7$.

The only remaining case is when $u_5$ is a $5p$-vertex and $u_3,u_7$ are $6p$-vertices.
Since $G-u_4$ is a graph with fewer edges than $G$ and the number of $3^+$-vertices did not increase, there is a $(3, 4)$-coloring $\varphi$ of $G-u_4$.
By Lemma~\ref{vertex-degree} $(i)$, $u_3$ and $u_5$ is $4$-saturated and $3$-saturated, respectively.
In particular, $\varphi(u_3)=\varphi(u_2)=4$ and $\varphi(u_5)=\varphi(u_6)=3$.
This further implies that $u_1$ is $3$-saturated and $u_7$ is $4$-saturated.
Now, recoloring $u_2, u_1, u_7$ with the color $3, 4, 3$, respectively, and coloring $u_4$ with the color $4$ gives a $(3, 4)$-coloring of $G$, which is a contradiction.
\end{proof}

\section{Discharging}


For each element $x\in V(G)\cup F(G)$, let $\mu(x)$ and $\mu^*(x)$ denote the {\em initial charge} and {\em final charge}, respectively, of $x$. Let $\mu(x)=d(x)-4$, so by Euler's formula, $$\sum_{x\in V(G)\cup F(G)} \mu(x)=-8.$$

\begin{figure}[ht]
\begin{center}
\includegraphics[scale=0.55]{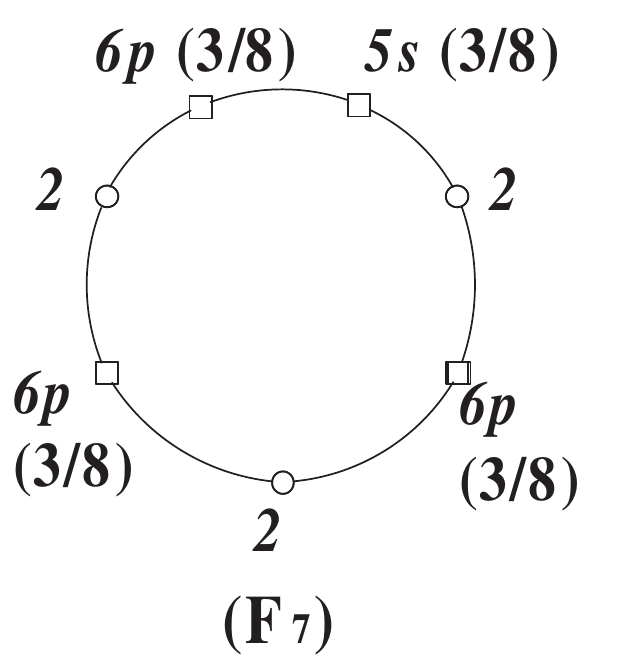}
\includegraphics[scale=0.55]{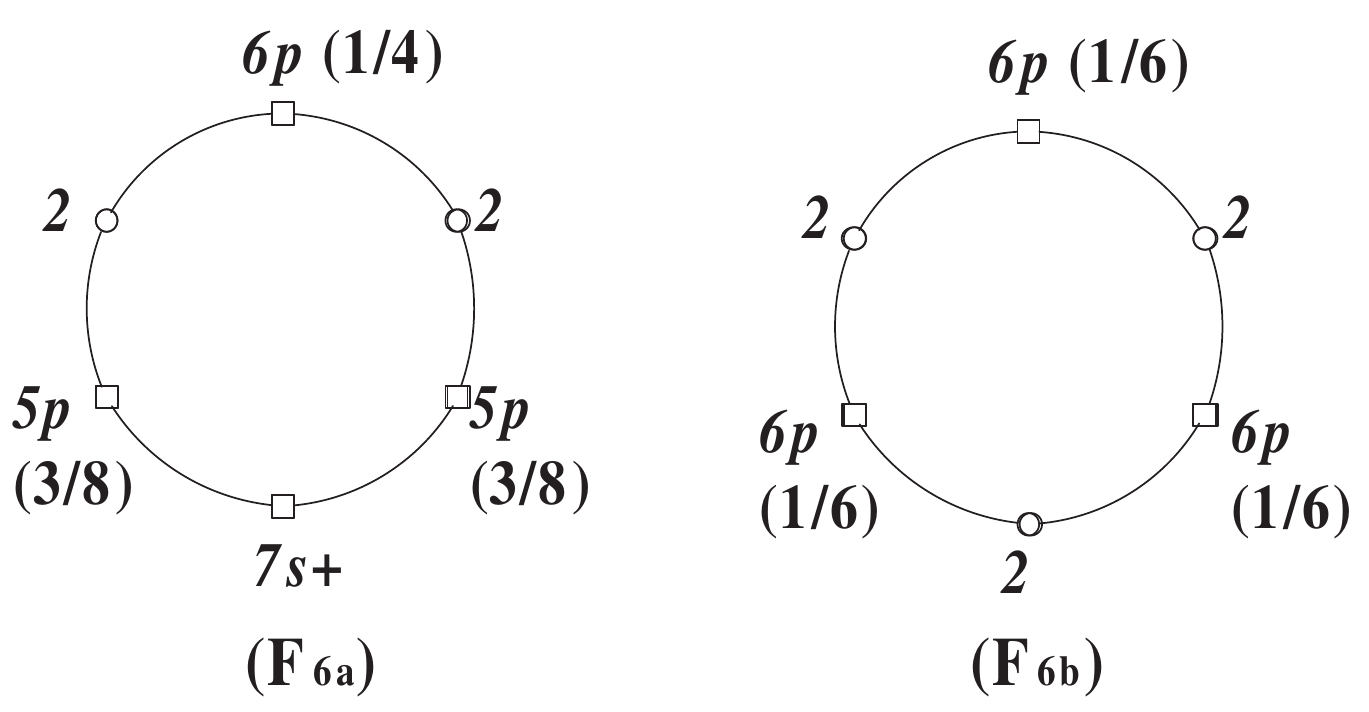}\\
\includegraphics[scale=0.55]{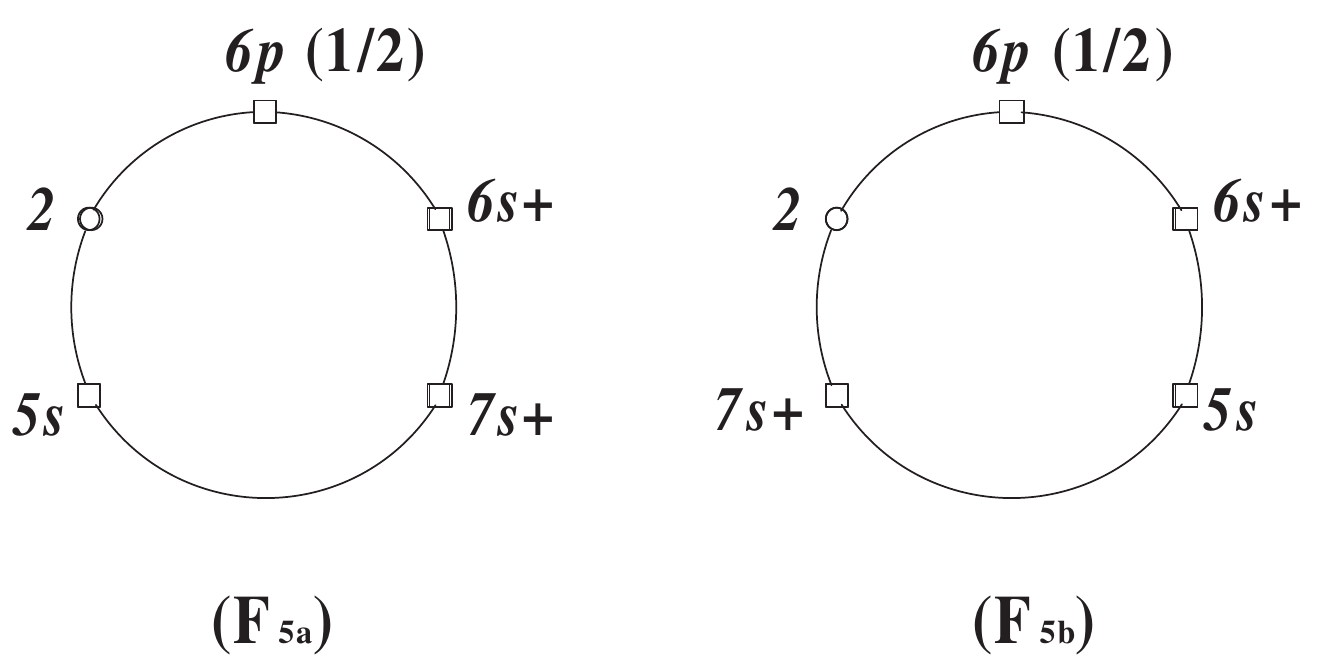}
\includegraphics[scale=0.55]{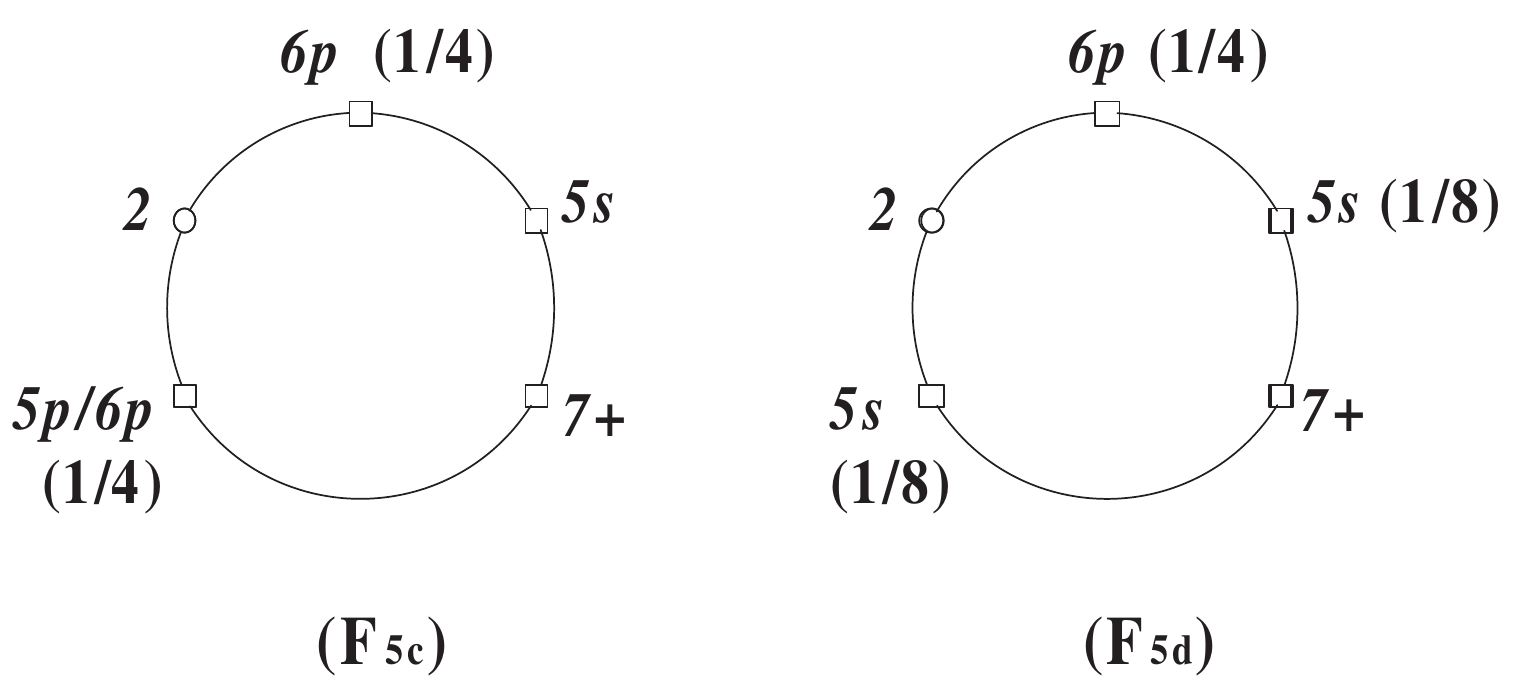}
\caption{Special  $7$-face, 6-faces, and 5-faces}
\label{figure-special-6-face}
\end{center}
\end{figure}

Here are the discharging rules:

 \begin{enumerate}[(R1)]
\item Let $v$ be a $5^+$-vertex. Then $v$ gives $\frac{1}{2}$ to each adjacent $2$-vertex; moreover,
\begin{enumerate}
\item if $d(v)\ge 8$, then $v$ gives $\frac{1}{2}$ to each adjacent $5p$-, $5s$-, $6p$-vertex and incident heavy edge;

\item if $d(v)=7$, then $v$ first gives $\frac{1}{2}$ to each adjacent $5p$-vertex and $6p$-vertex, then distributes its remaining charge evenly to adjacent $5s$-vertices and incident heavy edges;

\item if $d(v)\in \{5,6\}$, then $v$ distributes its remaining charge evenly to its adjacent $5p$-vertices and incident heavy edges.
\end{enumerate}

\item A heavy edge distributes its charge evenly to the two incident faces.

\item Let $f$ be a $5^+$-face. Then $f$ gives $\frac{1}{2}$ to each encountered incident $2$-vertex on a boundary walk of $f$; moreover,
\begin{enumerate}
\item if $d(f)\ge 8$, then $f$ gives $\frac{1}{2}$ to each encountered incident $5p$-, $5s$-, and $6p$-vertex on a boundary walk of $f$. (This means that each cut-vertex on $f$ that is a $5p$-, $5s$-, or $6p$-vertex is given at least $1$ in total.)

\item if $d(f)=7$ and $f\not=F_7$, then $f$ first gives $\frac{1}{2}$ to each incident $5p$-vertex and $6p$-vertex, then distributes its remaining charge evenly to each incident $5s$-vertex (if any exist). If $f=F_7$,  then $f$ gives $\frac{3}{8}$ to each incident $5s$-vertex and $6p$-vertex;

\item if $d(f)=6$ and $f\not=\{F_{6a}, F_{6b}\}$, then $f$ first gives $\frac{1}{2}$ to each incident $5p$-vertex and $\frac{1}{4}$ to each incident $5s$-vertex and $6p$-vertex, then distributes its remaining charge evenly to incident $5s$- and $6p$-vertices (if any exist);  if $f=F_{6a}$, then $f$ gives $\frac{3}{8}$ to each incident $5p$-vertex and  $\frac{1}{4}$ to the incident $6p$-vertex; if $f=F_{6b}$, then $f$ gives $\frac{1}{6}$ to each incident $6p$-vertex;

\item for the case $d(f)=5$, if $f$ is incident with two $2$-vertices, then it distributes  its charge evenly to each incident $5p$-, $5s$-, and $6p$-vertex (if any exist); if $f$ has at most one $2$-vertex and $f\not\in \{F_{5a},F_{5b},F_{5c}, F_{5d}\}$, then $f$ first gives $\frac{1}{4}$ to each incident $5p$-, $5s$-, and $6p$-vertex, then it distributes its remaining charge evenly to each incident $5p$-, $5s$-, and $6p$-vertex (if any exist); if $f\in \{F_{5a}, F_{5b}\}$, then it gives $\frac{1}{2}$ to the incident $6p$-vertex and its remaining charge to the $5s$-vertex; if $f\in \{F_{5c}, F_{5d}\}$, then it gives $\frac{1}{4}$ to each incident $5p$-vertex and $6p$-vertex, and its remaining charge evenly to incident $5s$-vertices.
\end{enumerate}
 \end{enumerate}

\begin{lem}\label{lem-face-charge}
If $f$ is a $5^+$-face, then $\mu^*(f)\geq 0$.
\end{lem}
\begin{proof}
We show that each face has nonnegative charge after the required charges by (R3).

If $f$ is a $5$-face, then $\mu(f)=1$ and $f$ is incident with at most two  $2$-vertices.
Clearly, $\mu^*(f)\geq 1-1=0$ by (R3d) and Lemma~\ref{lem:5-face-1-2vtx}.

If $f$ is a $6$-face, then $\mu(f)=2$ and $f$ is incident with at most three $2$-vertices by Lemma~\ref{lem:edge}.

Case 1: $f$ has at most one incident 2-vertex.  By Lemma~\ref{lem:6face} (c) and (d), and (R3c), $\mu^*(f)\geq 2-\max\{\frac{1}{2}\cdot 2+\frac{1}{4}\cdot 4,\frac{1}{4}\cdot6\}= 0$.

Case 2: $f$ has two incident 2-vertices.
 By Lemma~\ref{lem:6face} (b), $f$ has at most two incident $5p$-vertices. If $f$ has no incident $5p$-vertex, then $\mu^*(f)\geq 2-\frac{1}{2}\cdot 2-\frac{1}{4}\cdot 4=0$. If $f$ has one incident $5p$-vertex, then $f$ has at most two of $5s$-vertices and $6p$-vertices by Lemma~\ref{lem:6face} (b), so $\mu^*(f)\geq 2-\frac{1}{2}\cdot 3-\frac{1}{4}\cdot 2=0$. If $f$ has two incident $5p$-vertices, then $f$ is either a special face $F_{6a}$ or has neither $5s$-vertices nor $6p$-vertices.
Therefore, $\mu^*(f)\geq 2-\frac{1}{2}\cdot 2-\frac{3}{8}\cdot 2-\frac{1}{4}=0$ or $\mu^*(f)\geq 2-\frac{1}{2}\cdot 4=0$.

Case 3: $f$ has three incident 2-vertices.
If $f$ is incident with a $5$-vertex, then the other two vertices on $f$ are $7^+$-vertices by Lemma~\ref{lem:6face} (a), so $\mu^*(f)\geq 2-\max\{\frac{1}{2}\cdot 4,\frac{1}{2}\cdot 3+\frac{1}{4}\}=0$.
If $f$ is a special face $F_{6b}$ in Figure~\ref{figure-special-6-face}, then $\mu^*(f)\geq 2-\frac{1}{2}\cdot 3-\frac{1}{6}\cdot 3=0$.
If $f$ is not $F_{6b}$, then $\mu^*(f)\geq 2-\frac{1}{2}\cdot 3-\frac{1}{4}\cdot 2=0$.

If $f$ is a $7$-face, then $\mu(f)=3$. By (R3b) and Lemma~\ref{lem:7-face}, $\mu^*(f)\geq 7-4-\max\{\frac{1}{2}\cdot 6, \frac{1}{2}\cdot 5, \frac{1}{2}\cdot 3+\frac{3}{8}\cdot 4\}=0$.
If $f$ is a $8^+$-face, then $\mu^*(f)\ge d(f)-4-\frac{1}{2}d(f)\ge 0$ by (R3a).
\end{proof}

Now we consider the final charge of an arbitrary vertex.  Note that if a $5p$-, $5s$-, or $6p$-vertex is a cut vertex, then it must be visited more than once on a boundary walk of a $8^+$-face, thus it gets at least $1$ from the face by (R3). Therefore in the following three lemmas, we always assume that a $5p$-vertex or $5s$-vertex is in five different faces and a $6p$-vertex is in six different faces.

\begin{lem}
If $u$ is a $5p$-vertex, then $\mu^*(u)\geq 0$.
\end{lem}

\begin{proof}
By (R1), $u$ gives out $4\cdot \frac{1}{2}=2$ to its adjacent $2$-vertices. To show $\mu^*(u)\ge 0$, we need to prove that $u$ receives at least $1$ by the discharging rules.
Let $N(u)=\{u_0, v_1, v_2, v_3, v_4\}$ where $d(u_0)>2$ and $d(v_i)=2$ for $i\in[4]$.
For $i\in[4]$, let  $u_i$ be the  neighbor of $v_i$ that is not $u$.
We assume that the five faces incident with $u$ are $A,B,C,D,E$ as shown in Figure~\ref{figure-5p}.

\begin{figure}[ht]
\begin{center}
\includegraphics[scale=0.55]{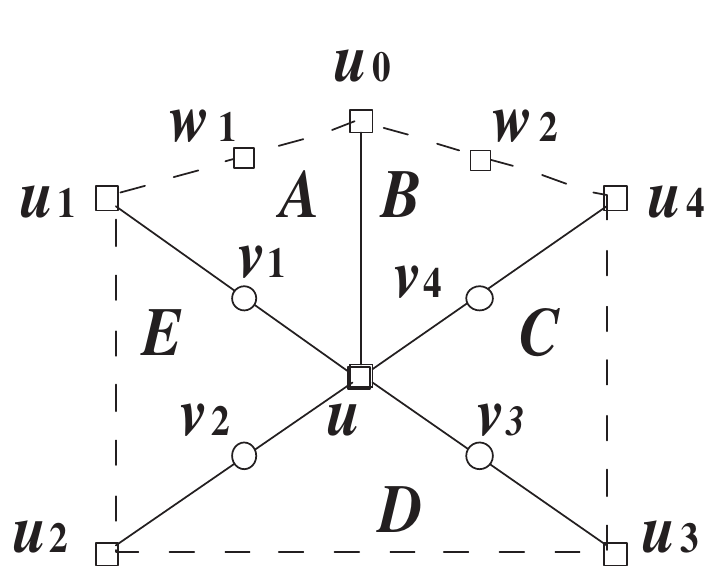}\hskip 1in
\includegraphics[scale=0.5]{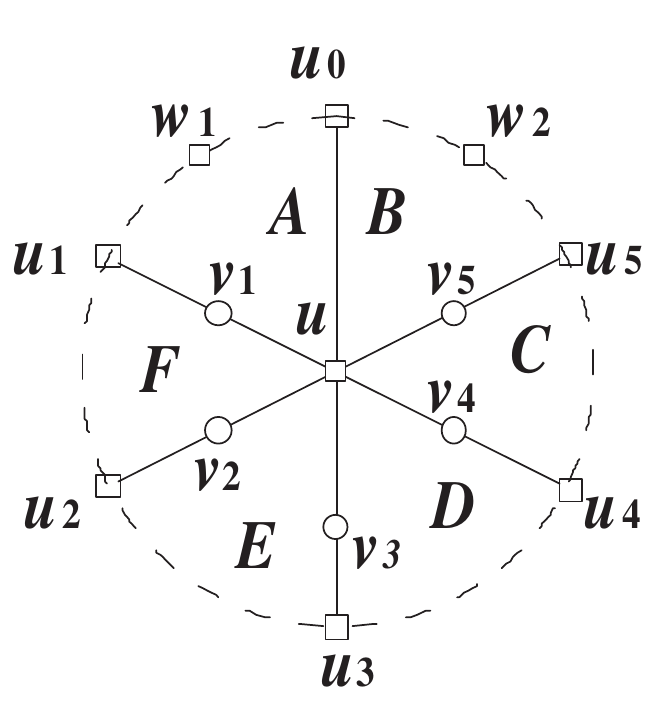}
\caption{A $5p$-vertex $u$ incident with five $5^+$-faces and a $6p$-vertex $u$ incident with six $5^+$-faces.}
\label{figure-5p}
\end{center}
\end{figure}

To get some idea regarding the degrees of the vertices on the five faces incident with $u$, we consider a $(3,4)$-coloring $\varphi$ of $G-u$, which exists since the number of edges decreased and the number of $3^+$-vertices did not increase.
By Lemma~\ref{vertex-degree} $(ii)$, $u_0$ is a $4$-saturated $6^+$-vertex and the four $2$-neighbors of $u$ are colored with the color $3$.
Since $u_0$ is non-recolorable, if $d(u_0)\leq 8$, then $u_0$ has a $3$-saturated neighbor and four neighbors of color $4$.
Furthermore, since no neighbor of $u$ is recolorable, for $i\in[4]$, $u_i$ is a $4$-saturated $6^+$-neighbor and if $d(u_i)\leq 8$, then $u_i$ has a $3$-saturated neighbor.


{\bf Case 1.} $u$ is incident with a special $6$-face $F_{6a}$.

By the ordering of the degrees of the vertices on $F_{6a}$, the special $6$-face must be either $A$ or $B$.
Without loss of generality, assume that $A$ is a special $6$-face $F_{6a}$ so that $u_1$ is a $6p$-vertex and $u_0$ is a $7s^+$-vertex.
As both $u_1$ and $u_2$ are $4$-saturated, and $u_1$ is adjacent to a $3$-saturated vertex, we conclude that $u_1$ cannot be adjacent to $u_2$.
Otherwise, $u_1$ has two $3^+$-neighbors, which implies that $u_1$ is not a poor vertex.
Hence, $E$ is a $6^+$-face.
By (R3), $u$ gets $\frac{1}{2}$ from $E$ and $\frac{3}{8}$ from $A$, and by (R1), $u$ gets $\frac{1}{2}$ from $u_0$.
So $u$ gets at least $1$ in total, as desired.


{\bf Case 2.} $u$ is not incident with a special 6-face and either $A$ or $B$ is a non-special $6^+$-face.

Note that by (R3), $u$ receives at least $\frac{1}{2}$ from each of its incident $6^+$-faces that are not special. So we may assume that $u$ is incident with exactly one $6^+$-face and four $5$-faces.
Without loss of generality, assume that $A$ is a $6^+$-face and let $B=uu_0w_2u_4v_4$. Note that $u_4$ is non-recolorable, which means that it has either a $3$-saturated neighbor or at least four neighbors colored with $3$. Since a $3^+$-neighbor $u_3$ of $u_4$ is $4$-saturated, we know that $u_4$ cannot be a $6p$-vertex.
Therefore, $B$ is not a special $5$-face.

\begin{enumerate}
\item We may assume that $u_0$ is a $6$-vertex. For otherwise, $u$ also gets $\frac{1}{2}$ from $u_0$ by (R1), thus $u$ gets at least $1$ in total.

\item We may assume that $w_2$ is not a $2$-vertex.  For otherwise, as $\varphi(u_0)=\varphi(u_4)=4$, $w_2$ is colored or can be recolored with $3$, and then $u_0$ must be a $7^+$-vertex, which contradicts (1).

\item We may assume that $w_2$ is a $5s$-vertex. For otherwise, none of $u_0,w_2,u_4$ is a $5p$-, $5s$-, or $6p$-vertex, so $u$ receives at least $\frac{1}{2}$ from $B$ by (R3d), thus $u$ get at least $1$ in total.

\item We may assume that each of $u_3$ and $u_2$ is either a $6r^+$-vertex or a $9^+$-vertex, and $u_1$ is either a $6s^+$-vertex or $9^+$-vertex.  For $z\in\{u_3,u_2,u_1\}$, observe that each $z$ must have either a $3$-saturated neighbor (other than $u_i$s) or four neighbors colored with 3 (other than $v_i$s).

\item We may assume that $u_4$ is either a $8^+$-vertex or a $7r$-vertex.
It must be that $\varphi(w_2)=3$, for otherwise, we can recolor $w_2$ with the color $3$ and color $u$ with $4$ to obtain a $(3, 4)$-coloring of $G$, which is a contradiction.
If $d(u_4)\leq 7$, then $u_4$ must have a $3$-saturated neighbor that is not $w_2$, for otherwise, we could recolor $u_0, w_2, u_4$ with $3,4,3$, respectively, and then color $u$ with the color $4$, a contradiction.
This implies that $u_4$ is a $7$-vertex and has at least three $5^+$-neighbors, i.e. $w_2, u_3$ and another $5s^+$- or $9^+$-neighbor, so $u_4$ is a $7r$-vertex.
\end{enumerate}

Now, $u_4u_3, u_3u_2, u_2u_1$ are all heavy edges.
Since each of $u_3$ and $u_2$ has at least two $5^+$-neighbors that are not $5p$-, $5s$- and $6p$-vertices, by (R1), the heavy edges $u_4u_3, u_3u_2, u_2u_1$ get at least $\frac{1}{3}+\frac{1}{6}$, $\frac{1}{6}\cdot 2$,  $\frac{1}{6}$, respectively, from $u_4, u_3, u_2$.
By (R2) and (R3d), $u$ receives at least $\frac{1}{2}(\frac{1}{3}+\frac{1}{6}+\frac{1}{6}\cdot 2+\frac{1}{6})=\frac{1}{2}$ from faces $C,D,E$, and thus a total of $1$, as desired.

{\bf Case 3.} $u$ is not incident with a special $6$-face and both $A$ and $B$ are $5$-faces.

Let $A=uu_0w_1u_1v_1$ and $B=uu_0w_2u_4v_4$.

\begin{enumerate}
\item If $k$ is the number of vertices in $\{w_1, w_2\}$ that is either a $2$-vertex or a $5s$-vertex, then $d(u_0)\geq 6+k$.  This is because if $w_i$ is either a $2$-vertex or a $5s$-vertex, then $\varphi(w_i)=3$, otherwise we can recolor $w_i$ with $3$ and color $u$ with $4$ to obtain a $(3, 4)$-coloring of $G$, which is a contradiction. The lower bound on $d(u_0)$ follows since $u_0$ is $4$-saturated and cannot be recolored with the color 3.

\item We may assume that $C,D,E$ are $5$-faces. For otherwise, $u$ gets at least $\frac{1}{2}$ from an incident $6^+$-face by (R3). Now, if $d(u_0)\ge 7$, then $u$ gets another $\frac{1}{2}$ from $u_0$ by (R1), for a total of $1$.  If $d(u_0)= 6$, then each of $w_1$ and $w_2$ is neither $2$-vertex nor $5s$-vertex by (1), and thus each of $u_0,w_1,w_2$ is not a $2/5p/5s/6p$-vertex.
By (R3d), $u$ gets at least $\frac{1}{4}\cdot 2$ from $A$ and $B$, for a total of $1$.

\item We observe each of $u_3$ and $u_2$ is either a $9^+$-vertex or a $6r^+$-vertex.  This follows from the fact that each of  $u_3$ and $u_2$ is $4$-saturated, has two $4$-saturated neighbors, and is not recolorable with $3$ (which implies either a $3$-saturated neighbor or at least four 3-colored neighbors other than $v_2$ and $v_3$).

\item Assume $d(u_1), d(u_4)\leq 8$. For $i\in[2]$, $u_{3i-2}$ is a $7s^+$-vertex if $d(w_i)=2$ and is a $6s^+$-vertex if $d(w_i)\geq 3$.
\end{enumerate}

Now, $u_1u_2, u_2u_3, u_3u_4$ are all heavy edges.  By (R1), $u_2$ sends at least $\min\left\{\frac{6-4-3\cdot\frac{1}{2}}{3}, \frac{7-4-5\cdot \frac{1}{2}}{2}, \frac{1}{2}\right\}= \frac{1}{6}$ to each of $u_1u_2$ and $u_2u_3$, and likewise, $u_3$ sends at least $\frac{1}{6}$ to each of $u_2u_3$ and $u_3u_4$.
\begin{itemize}
\item Assume that both $w_1$ and $w_2$ are $2$-vertices.
Now, $u_0$ is an $8^+$-vertex and gives $\frac{1}{2}$ to $u$ by (R1a).
Also, $u_1$ ($u_4$, respectively) is a $7s^+$-vertex and gives at least $\frac{7-4-5\cdot\frac{1}{2}}{2}=\frac{1}{4}$ to the heavy edge $u_1u_2$ ($u_3u_4$, respectively) by (R1). By (R2) and (R3), $C,D,E$ give at least $\frac{1}{2}(\frac{1}{6}\cdot 4+\frac{1}{4}\cdot 2)>\frac{1}{2}$ to $u$.

\item Without loss of generality, assume that $w_1$ is a $2$-vertex and $w_2$ is a $3^+$-vertex.   Now, $u_0$ is a $7^+$-vertex and gives $\frac{1}{2}$ to $u$ by (R1), and the $5$-face $B$ gives at least $\frac{1}{4}$ to $u$ by (R3).  Then $u$ gets at least $\frac{1}{2}+\frac{1}{4}$ from $u_0$ and $B$, and at least $\frac{1}{2}(\frac{1}{6}\cdot 4)>\frac{1}{4}$ from $C,D,E$ by (R3).

\item Finally, assume that neither of $w_1,w_2$ is a $2$-vertex. By (R3d), each of $A,B$ gives at least $\frac{1}{4}$ to $u$. Furthermore, each of $A,B$ gives at least $\frac{1}{2}$ to $u$ if neither $w_1$ nor $w_2$ is a $5s$-vertex. (In the case, each of $A,B$ cannot contain any $2/5p/5s/6p$-vertex other than $u,v_1,v_4$ because each of $u_0,u_1,u_4$ is $4$-saturated, non-recolorable and has at least two $3^+$-neighbors.)   Now if either  $w_1$ or $w_2$ is a $5s$-vertex, then $u_0$ is a $7^+$-vertex, and thus $u_0$ gives $\frac{1}{2}$ to $u$ so $u$ gets a total of $\frac{1}{4}\cdot2+\frac{1}{2}\geq 1$.
\end{itemize}
Hence, $u$ always gets at least $1$, as desired.
\end{proof}

\begin{lem}
If $u$ is a $6p$-vertex, then $\mu^*(u)\geq 0$.
\end{lem}

\begin{proof}
The initial charge of $u$ is $2$, and by (R1c), $u$ gives out $\frac{1}{2}\cdot 5$ to its $2$-neighbors.
To show $\mu^*(u)\ge 0$, we need to prove that $u$ receives $\frac{1}{2}$ by the discharging rules.

Let $N(u)=\{u_0, v_i: i\in[5]\}$ where  $d(u_0)>2$ and $d(v_i)=2$ for $i\in [5]$. For $i\in [5]$, let $u_i$ be the neighbor of $v_i$ that is not $u$. We assume that the six faces incident with $u$ are $A, B, C, D, E, F$ as shown in Figure~\ref{figure-5p}.

Since $G-u$ is a graph with fewer edges than $G$ and the number of $3^+$-vertices did not increase, there exists a $(3,4)$-coloring $\varphi$ of $G-u$.
By Lemma~\ref{vertex-degree}, either $\varphi(v_i)=4$ for $i\in[5]$ and $u_0$ is $3$-saturated and non-recolorable, or at least four of $v_i$'s are colored with $3$ and $u_0$ is $4$-saturated and non-recolorable.
In the former case, $u_i$ with $i\in[5]$ are $3$-saturated and non-recolorable, and in the latter case, at least four of the $u_i$'s are $4$-saturated and non-recolorable.

\begin{enumerate}

\item\label{i-u0} We may assume that $d(u_0)\le 6$.   By (R1), $u$ gets $\frac{1}{2}$ from $u_0$ if $d(u_0)\ge 7$.

\item Also, we may assume that $u$ is not incident with a special face $F_{6b}$.

\noindent
If $u$ is incident with $F_{6b}$, then by Lemma~\ref{lem:F2}, $u$ is also incident with two other faces where each face is not a $5$-face with two $2$-vertices.
By (R3), each face that is not a 5-face with two $2$-vertices sends at least $\frac{1}{6}$ to $u$, plus $F_{6b}$ sends $\frac{1}{6}$ to $u$.
Thus, $u$ gets a total of at least $\frac{1}{2}$.

\item We may assume that $A$ is a $5$-face with two $2$-vertices.

\noindent
By (R3), each face that is neither a 5-face with two 2-vertices nor $F_{6b}$ gives at least $\frac{1}{4}$ to $u$, so we may assume that one of $A$ and $B$, say $A$, must be a $5$-face with two $2$-vertices.

\item $B$ is not a $5$-face with two $2$-vertices, and furthermore we may assume that $C,D,E,F$ are $5$-faces with two 2-vertices.

\noindent
Suppose that $B$ is a $5$-face with two $2$-vertices, so that both $w_1$ and $w_2$ are $2$-vertices.
If $u_0$ is 3-saturated, then both $u_1$ and $u_5$ are $3$-saturated. So $w_1,w_2$ are colored or can be recolored with $4$. This implies that $u_0$ is a $7^+$-vertex, which contradicts~(\ref{i-u0}).
If $u_0$ is $4$-saturated, then either $u_1$ or $u_5$ is $4$-saturated.
Without loss of generality assume that $u_1$ is $4$-saturated, so either $\varphi(w_1)=3$ or $w_1$ can be recolored with $3$. This implies that $u_0$ is a $7^+$-vertex, which contradicts~(\ref{i-u0}).

Now $u$ receives at least $\frac{1}{4}$ from $B$ by (R3). We may assume that $C,D,E,F$ are 5-faces with two 2-vertices, for otherwise, $u$ receives another $\frac{1}{4}$ to get a total of at least $\frac{1}{2}$.

\item $B$ is not a special face $F_7$.

\noindent Without loss of generality, assume that $B$ is a special face $F_7$, which sends $\frac{3}{8}$ to $u$.  This implies that $u_0$ is a $5s$-vertex, which further implies that $u_0$ is 3-saturated and $u_i$ is 3-saturated for each $i\in[5]$. Note that $A$ is a $5$-face with two $2$-vertices. Since $u_0$ is non-recolorable, it has a $4$-saturated ($6s^+$- or $9^+$-)neighbor, which means that $\varphi(w_1)=3$. Recolor $w_1$ with $4$ and then color $u$ with $3$, we obtain a $(3, 4)$-coloring of $G$, a contradiction.


\item\label{i-5Fab} We may assume that $B$ is a $6^-$-face. Moreover, if $B$ is a 5-face, then it can be neither $F_{5a}$ nor $F_{5b}$.
 Otherwise, $u$ receives at least $\frac{1}{2}$ by (R3a), (R3b), (R3d).

\item\label{i-6face} $B$ must be a $6$-face.

\noindent
Suppose otherwise. From above, assume that $B$ is a 5-face with at most one 2-vertex.
Note that $B$ must have exactly one 2-vertex since $v_5$ is a 2-vertex.
By (R3), $B$ gives $u$ at least $\frac{1}{2}$ if $u$ is the only $5p$-, $5s$-, or $6p$-vertex on $B$. So consider the case when $B$ is a $5$-face with one $2$-vertex $v_5$ and at least two $5p$-, $5s$-, or $6p$-vertices.
Note that none of $u_0,w_2,u_5$ can be a $6p$- or $5p$-vertex.

Assume that $u_0$ is $3$-saturated.  Then $\varphi(v_i)=4$ and $u_i$ is $3$-saturated for $i\in[5]$.  The 2-vertex $w_1$ is colored or can be recolored with $4$. Therefore $u_0$ is a $6s^+$-vertex.
Thus, either $u_5$ or $w_2$ is a $5s$-vertex.
Since $B$ is not $F_{5a}$ or $F_{5b}$ by (\ref{i-5Fab}), when one of $u_5$ and $w_2$ is a $5s$-vertex, the other one is a $6^-$-vertex.
Now if $w_2$ is a $5s$-vertex, then $w_2$ is colored or can be recolored with $4$ without making $w_2$ $4$-saturated, so $u_0$ must have another $4$-saturated neighbor. Thus, $d(u_0)\ge 7$, which contradicts (\ref{i-u0}).
If $u_5$ is a $5s$-vertex, then $w_2$ must be the $4$-saturated neighbor of $u_5$ and $u_0$.
Thus, we can recolor $u_0,u_5$ with $4$ and $w_2$ with $3$, and color $u$ with $3$ to obtain a $(3, 4)$-coloring of $G$, which is a contradiction.

Assume that $u_0$ is $4$-saturated.  Then $u_0$ is a $6s^+$-vertex. Now, since $d(u_0)\leq 6$, the $2$-vertex $w_1$ cannot be colored or recolored with $3$. This implies that $\varphi(w_1)=4$ and $\varphi(u_1)=3$, and moreover, $\varphi(v_1)=4$.
Furthermore, for $i\in[5]-\{1\}$, $\varphi(v_i)=3$ and $u_i$ is $4$-saturated. Since $u_5$ is $4$-saturated, it is a $6s^+$-vertex.
So $w_2$ is a $5s$-vertex, and $\varphi(w_2)=3$ or $w_2$ can be recolored with $3$. Again, since $B$ is neither $F_{5a}$ nor $F_{5b}$, we know $d(u_5)\le 6$. Then $w_2$ is the only $3$-saturated neighbor of $u_0$ and $u_5$. So by recoloring $u_0,w_2, u_5$ with $3,4,3$, and coloring $u$ with $4$, we obtain a $(3, 4)$-coloring of $G$, which is a contradiction.

From now on, denote $B=uu_0w_2w_2'u_5v_5$.
\item\label{i-2vx} Either $w_2$ or $w_2'$ is a $2$-vertex.

\noindent
For otherwise, $B$ contains exactly one $2$-vertex $v_5$. Moreover, the only $6p^-$-vertex that $B$ contains is $u$. We may assume that $B$ contains at least three $5s$-vertices, for otherwise $u$ gets at least $\frac{1}{3}\cdot( 6-4-\frac{1}{2})=\frac{1}{2}$ from $B$ by (R3c). Since no $5s$-vertex can be adjacent to two $5s$-vertices, by Lemma~\ref{vertex-degree}, we know either $w_2$ or $w_2'$ is not a $5s$-vertex, and both $u_0$ and $u_5$ are $5s$-vertices.
Now, both $u_0$ and $u_5$ cannot be $4$-saturated, thus they are both $3$-saturated. Moreover, $\varphi(v_i)=4$ and $\varphi(u_i)=3$ for $i\in [5]$.  Now we can recolor $w_1$ with $4$ and color $u$ with $3$ to obtain a $(3, 4)$-coloring of $G$, which is a contradiction.

\item $u_0$ cannot be a $5p/5s/6p$-vertex.

As $u$ is a $6p$-vertex, $u_0$ must be a $5s^+$- or $9^+$-vertex by Lemma~\ref{vertex-degree} (iii). If $u_0$ is a $5s$-vertex, then $u_0$ is $3$-saturated,  thus $\varphi(v_i)=4$ and $\varphi(u_i)=3$ for each $i\in [5]$. This means that the $2$-vertex $w_1$ is colored or can be recolored with $4$. Since $u_0$ is non-recolorable, it has a $4$-saturated neighbor and thus $u_0$ is a $6^+$-vertex, a contradiction.





\item $u_5$ cannot be a $5p$- or $6p$-vertex.

Suppose otherwise.  Then $w_2'$ is a $2$-vertex and $u_4$ is the unique $3^+$-neighbor of $u_5$.  Since $G-v_5$ has fewer edges than $G$, it has a $(3,4)$-coloring $\phi$.  Then $u_5$ and $u$ are respectively $3$- and $4$-saturated and both are non-recolorable.

First let $u_5$ be $3$-saturated. Then $u_4$ is the unique $4$-saturated neighbor of $u_5$. So $v_4$ can be recolored $3$. But $\phi(v_i)=\phi(u_0)=4$ for $i\in [3]$, to make $u$ $4$-saturated. So we can recolor $u$ with $3$, a contradiction.

Now assume that $u_5$ is $4$-saturated, which means that $u_5$ is a $6p$-vertex. Then $u_4$ is the unique $3$-saturated neighbor of $u_5$. Now $v_4$ can be recolored with $4$. So $u_0$ is $4$-saturated and $v_i$ for $i\in [3]$ is colored with $3$, which implies that $u_i$ for $i\in [3]$ is colored with $4$.  Then $w_1$ can be recolored with $3$. On the other hand, since $u_5$ is a $4$-saturated $6p$-vertex, $w_2'$ must be a $2$-vertex colored with $4$. Thus $w_2$ must be colored with $3$. (For otherwise, $w_2'$ can be recolored with $3$, a contradiction.) As $u_0$ is $4$-saturated, $d(u_0)\ge 7$, a contradiction.


\item If $u_5$ is a $5s$-vertex, then $u_5$ and $u$ are the only $5p/5s/6p$-vertices on $B$.

Let $u_5$ be a $5s$-vertex.  Again $G-v_5$ has a $(3,4)$-coloring $\phi$, in which $u_5$ is $3$-saturated and $u$ is $4$-saturated and both are non-recolorable, by Lemma~\ref{vertex-degree}. So $\phi(v_i)=4$ for $i\in [4]$ and $u_0$ is $3$-saturated.  Then $\phi(u_i)=3$ for $i\in [4]$. So the $2$-vertex $w_1$ can be recolored with $4$.

If $w_2'$ is a $2$-vertex, then $\phi(w_2')=3$ and $w_2$ must be $4$-saturated.  In this case, $w_2$ is the only $4$-saturated neighbor of $u_0$ (note that $d(u_0)\le 6$), so $w_2$ cannot be a $6p$- or $5s$-vertex.  So $u_5$ and $u$ are the only $5p/5s/6p$-vertices on $B$, as desired.

So $w_2$ is a $2$-vertex and $w_2'$ is the $4$-saturated neighbor of $u_5$. We claim $\phi(w_2)=3$, for otherwise, $u_0$ has to be a $7^+$-vertex to be $3$-saturated and non-recolorable.  Now $w_2'$ is non-recolorable and $4$-saturated, it must be a $6s^+$-vertex or an $8^+$-vertex. So again, $u_5$ and $u$ are the only $5p/5s/6p$-vertices on $B$, as desired.

\end{enumerate}

By (9)-(11), $B$ contains at most two $5p/5s/6p$-vertices, thus $u$ receives at least $\frac{1}{2}$ from $B$ by (R3c), as desired.
\end{proof}

\begin{lem}
If $u$ is a $5s$-vertex, then $\mu^*(u)\geq 0$.
\end{lem}

\begin{proof}
The initial charge of $u$ is $1$, and by (R1c), $u$ gives out $\frac{1}{2}\cdot 3$ to its $2$-neighbors.
To show $\mu^*(u)\ge 0$, we need to prove that $u$ receives $\frac{1}{2}$ by the discharging rules.

Let $N(u)=\{x,y,v_1,v_2,v_3\}$ with $d(x),d(y)>2$ and $d(v_i)=2$ and let $u_i$ be the other neighbor of $v_i$ for $i\in[3]$. Depending on whether $x,y,u$ are on the same face or not, we could have two different embeddings around $u$ (see Figure~\ref{figure-5s}).

\begin{figure}[ht]
\begin{center}
\includegraphics[scale=0.45]{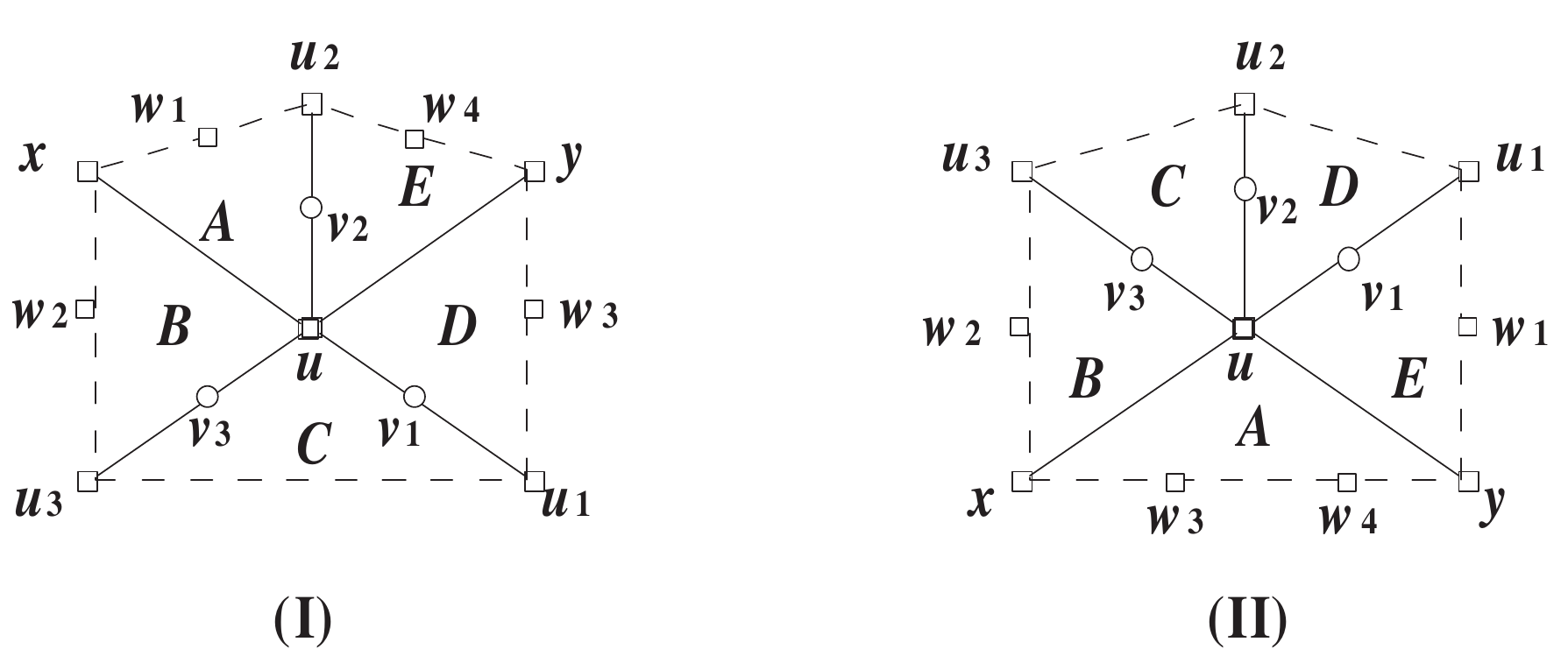}
\caption{Two possible embeddings containing $5s$-vertex $u$ with five $5$-faces.}
\label{figure-5s}
\end{center}
\end{figure}

Since $G-v_2$ is a graph with fewer edges than $G$ and the number of $3^+$-vertices did not increase, there exists a $(3,4)$-coloring $\varphi$ of $G-v_2$.
By Lemma~\ref{vertex-degree},  $u$ is $3$-saturated  and $u_2$ is $4$-saturated, and both are non-recolorable. Without loss of generality, we may assume that $x$ is $4$-saturated and $\varphi(y)=\varphi(v_1)=\varphi(v_3)=3$. Also, $u_1$ and $u_3$ are $4$-saturated and non-recolorable.


We may assume $d(x),d(y)\le 7$, for otherwise, $u$ gets at least $\frac{1}{2}$ by (R1a).
 Moreover, by Lemma~\ref{lem:7-face} and (R3), $u$ receives at least $\frac{1}{4}$ from each incident $6^+$-face, so we assume that $u$ is incident with at most one $6^+$-face.


{\bf Case 1.} $x,y,u$ are not on the same face (see Figure~\ref{figure-5s} (I) for an illustration). We let $w_1, w_2$ be the neighbors of $x$ on $A$ and $B$, respectively.
\begin{enumerate}
\item $u$ is not incident with a special $5$-face $F_{5b}$ or $F_{5c}$.  This follows from the fact that $u$ is adjacent to a 2-vertex on each incident face.

\item We may assume that neither of $B,D$ is a special $5$-face $F_{5a}$ or $F_{5d}$. It follows that neither of $B,D$ is a special $5$-face.

By symmetry, let $B$ be a special 5-face $F_{5a}$ or $F_{5d}$. Then $x$ is a $7s^+$-vertex, $w_2$ is a $5s^+$-vertex, and $u_3$ is a $6p$-vertex.  So $u_1u_3\not\in E(G)$, and $C$ must be a $6^+$-face.  By (R1), $u$ receives at least $\frac{1}{4}$ from $x$, and by (R3), $u$ receives at least $\frac{1}{4}$ from $C$.

\item We may assume that $d(w_1)>2$ and $d(w_2)=2$.

Suppose otherwise that $d(w_1)=2$. When $A$ is a $6^+$-face, $u$ receives at least $\frac{1}{4}$ from $A$ by Lemma~\ref{lem:7-face} and (R3). Moreover, $B$ must be a $5$-face with two $2$-vertices, for otherwise, $u$ receives another $\frac{1}{4}$ from $B$ by (2) and (R3d). So $w_2$ is a $2$-vertex and is colored or can be recolored with $3$. Since $x$ is $4$-saturated and non-recolorable, $x$ must be a $7s^+$-vertex with another $3$-saturated ($5s^+$- or $9^+$-)neighbor other than $u$. Thus $u$ receives at least $\frac{1}{4}$ from $x$ by (R1), and then $\frac{1}{2}$ in total. When $A$ is a $5$-face, $w_1$ is colored or can be recolored with $3$. Again, $x$ is a $7s^+$-vertex and gives $u$ at least $\frac{1}{4}$. Also, we may assume that $B$ is a $5$-face with two $2$-vertices, for otherwise, $u$ can receive another $\frac{1}{4}$ from $B$ by (2) and (R3). So $w_2$ is colored or can be recolored with $3$. This means that $x$ is an $8^+$-vertex, a contradiction.  Therefore, we assume that $d(w_1)>2$.

Now suppose that $d(w_2)>2$. By (R3), $u$ receives at least $\frac{1}{4}$ from $B$, since $B$ is a $6^+$-face or a $5$-face with only one $2$-vertex. Since $d(w_1)>2$, we may assume that $A$ is a special $5$-face, for otherwise, $u$ can receive another $\frac{1}{4}$ from $A$ by (R3). By (1), $A$ is an $F_{5a}$ or $F_{5d}$. Clearly, $x$ is a $7s^+$-vertex with at least two $5s^+$-neighbors. So $u$ receives $\frac{1}{4}$ from $x$ by (R1) as well. Thus we assume that $d(w_2)=2$.

\item We may assume that $A$ is a special $5$-face and $B,C,D,E$ are $5$-faces.

If $A$ is not a special $5$-face, then $u$ can receive at least $\frac{1}{4}$ from $A$ by (R3) since $A$ is a $6^+$-face or a $5$-face with only one $2$-vertex (note that $d(w_1)>2$). If $B$ is a $6^+$-face, $u$ receives another $\frac{1}{4}$ from $B$. Otherwise, $B$ is a $5$-face with two $2$-vertices because $d(w_2)=2$. So $w_2$ is colored or can be recolored with $3$. As $x$ is $4$-saturated and non-recolorable, $x$ must be a $7s^+$-vertex and can give $u$ at least $\frac{1}{4}$ by (R1).

By (1), $A$ is an $F_{5a}$ or $F_{5d}$.  Therefore, $x$ is a $7s^+$-vertex with at least two $5s^+$-neighbors. Then $u$ receives at least $\frac{1}{4}$ from $x$ by (R1). If one of $B,C,D,E$ is a $6^+$-face, then $u$ can receive another $\frac{1}{4}$ by (R3). Therefore we may assume that $B, C, D,E$ are $5$-faces.

\item We may assume that $A$ is an $F_{5a}$.

For otherwise, by (1) and (4), $A$ must be $F_{5d}$. Then $d(x)= 7$, $u_2$ is a $6p$-vertex and $w_1$ is a $5s$-vertex.  So $w_1$ is the only $3$-saturated neighbor of the $6p$-vertex $u_2$. Since $x$ has at least two $5s^+$-neighbors, $u$ gets at least $\frac{1}{4}$ from $x$ by (R1). We may assume that $B$ is a $5$-face with two $2$-vertices, for otherwise, $u$ receives another $\frac{1}{4}$ from $B$. Now that $w_2$ is colored or can be recolored with $3$. Since $x$ is $4$-saturated, non-recolorable and $d(x)=7$, $w_1$ is the only $3$-saturated neighbor of $x$. Thus we can recolor $x, w_1, u_2, u$ with $3, 4, 3, 4$, respectively, then color $v_2$ with $4$ to obtain a $(3, 4)$-coloring of $G$, which is a contradiction.
\end{enumerate}

By (5), $u_2,w_1,x$ are $6p$-, $6s^+$-, $7s^+$-vertices, respectively. Clearly, $x$ gives at least $\frac{1}{4}$ to $u$ and the heavy edge $xw_1$. By (2) and (4), both $B$ and $D$ are $5$-faces with two $2$-vertices, for otherwise, $u$ can receive another $\frac{1}{4}$ from $B$ or $D$. Then both of $w_2,w_3$ are $2$-vertices. Also, $C$ is a $5$-face. Note that $w_2$ is colored or can be recolored with $3$ and both of $u_1,u_3$ are $4$-saturated and non-recolorable. So each of $u_1,u_3$ has either a $3$-saturated neighbor or at least four neighbors colored with $3$. This implies that $u_1$ is a $6s^+$- or $8^+$-vertex, and $u_3$ is either a $8^+$-vertex or a $7s^+$-vertex with another $5s^+$- or $9^+$-neighbor distinct from $u_1$. So $u_3$ gives at least $\frac{1}{4}$ to the heavy edge $u_1u_3$. Then $u$ receives $\frac{1}{4}+\frac{1}{2}\cdot(\frac{1}{4}+\frac{1}{4})=\frac{1}{2}$ from $x$, $A$ and $C$ in total.

\medskip

{\bf Case 2.} $x,y,u$ are in the same face, denoted by $A$ (see for example Figure~\ref{figure-5s} (II)). Let $w_1$ be the neighbor of $y$ on $E$ and $w_2$ be the neighbor of $x$ on $B$.

\begin{enumerate}
\item $u$ must be incident with a special $5$-face $F_{5a}, F_{5b}, F_{5c}$ or $F_{5d}$.

Assume that $u$ has none of the special 5-faces.  By (R3), each $6^+$-face or $5$-face with at most one $2$-vertex gives at least $\frac{1}{4}$ to $u$, in particular, $A$ gives $\frac{1}{4}$ to $u$.  So all other faces are $5$-faces with two $2$-vertices. This implies that $d(w_1)=d(w_2)=2$.

Recall that $x$ and $u_3$ are $4$-saturated. Then $w_2$ is colored or can be recolored with $3$. Note that $d(x)\le 6$, for otherwise, $u$ gets $\frac{1}{4}$ from $x$.  So $u,w_2$ are the only neighbors of $x$ of color $3$.  Now recolor $x$ with $3$ and $u$ with $4$, and we can color $v_2$ with $3$, a contradiction.





\item None of $C,D,B,E$ is a special $5$-face.

Clearly $C,D$ cannot be special $5$-faces.  If $B$ or $E$ is a special $5$-face, then they only could be in $\{F_{5a}, F_{5d}\}$.   By symmetry, assume that $B$ is a special 5-face. Then $u_3, w_2, x$ are $6p$-, $5s^+$- and $7s^+$-vertices, respectively.  Since $u_3$ is a $6p$-vertex, $u_3u_2\not\in E(G)$, so $C$ is a $6^+$-face, thus $u$ gets at least $\frac{1}{4}$ from $C$ by (R3c).  So $u$ gets at least $\frac{1}{2}$ since $u$ gets at least $\frac{1}{4}$ from $x$ by (R1), a contradiction.


\item $A$ cannot be a special $5$-face.

Clearly, $A$ cannot be a special $5$-face $F_{5a}$. So we may assume that $A$ is a special $5$-face in $\{F_{5b}, F_{5c}, F_{5d}\}$. So it contains a $7^+$-vertex which is $x$ or $y$.  By (R1), $u$ receives at least $\frac{1}{4}$ from the $7^+$-vertex.   We may assume that $B,E$ are $5$-faces with two $2$-vertices (for otherwise, $u$ receives at least $\frac{1}{4}$ from them). Then $d(w_1)=d(w_2)=2$.

Note that $w_2$ is colored or can be recolored with $3$, since both $u_3$ and $x$ are $4$-saturated. Since $u_3$ is non-recolorable, it must have a $3$-saturated neighbor and four neighbors of color $4$, so $u_3$ must be a $7^+$-vertex. Note that $u_2$ must be a $6r^+$-vertex or an $8^+$-vertex and $u_1$ be a $6s^+$-vertex or an $8^+$-vertex, since they are all $4$-saturated and non-recolorable.  By (R1), $u_2$ gives at least $\frac{6-4-3\cdot \frac{1}{2}}{3}=\frac{1}{6}$ to each of the heavy edges $u_2u_3$ and $u_2u_1$, and $u_3$ gives at least $\frac{7-4-5\cdot \frac{1}{2}}{2}=\frac{1}{4}$ to the heavy edge $u_2u_3$.  So by (R2) and (R3), $u$ gets at least $\frac{1}{2}(\frac{1}{4}+\frac{1}{6}\cdot  2)>\frac{1}{4}$ from $C$ and $D$. So $u$ gets at least $\frac{1}{2}$.
\end{enumerate}
Now (2) and (3) contradict (1).
\end{proof}

\begin{lem}
Every vertex $u\in V(G)$ has $\mu^*(u)\geq 0$.
\end{lem}

\begin{proof}
We consider the cases according to the degree of $u$. Clearly, $\mu^*(u)=2-4+4\cdot\frac{1}{2}=0$ if $d(u)=2$ by the rules.  If $d(u)\ge 8$, then $\mu^*(u)\ge d(u)-4-d(u)\cdot \frac{1}{2}\ge 0$. If $d(u)=7$, then $u$ has at least one neighbor that is not a $2$-, $5p$-, or $6p$-vertex by Lemma~\ref{vertex-degree}. Thus by (R1b), $\mu^*(u)\ge 7-4- \frac{6}{2}\ge 0$
. For $d(u)\in \{5,6\}$, the lemmas have shown that $\mu^*(u)\ge 0$.  Note that $d(u)\not=3$ and $4$-vertices have initial and final charges $4-4=0$.
\end{proof}


\end{document}